\documentclass[12pt,oneside]{amsart}
\usepackage{amscd,a4,amssymb}
\usepackage{typearea,hyperref}
\usepackage[usenames,dvipsnames,svgnames,table]{xcolor}

\newcounter{lemmacounter}
\newcounter{thmcounter}

\newtheorem{lemma}[lemmacounter]{Lemma}
\newtheorem{proposition}[lemmacounter]{Proposition}
\newtheorem*{corollary*}{Corollary}

\newtheorem*{remark*}{Remark}
\newtheorem{theorem}[thmcounter]{Theorem}
\newtheorem*{theorem*}{Theorem}

\newcommand{\IG}{{\bf G}}
\newcommand{\IC}{{\bf C}}
\newcommand{\IR}{{\bf R}}
\newcommand{\IQbar}{\overline{\bf Q}}
\newcommand{\IZ}{{\bf Z}}
\newcommand{\IN}{{\bf N}}

\newcommand{\ssm}{{\smallsetminus}}

\newcommand{\height}[1]{h({#1})}
\newcommand{\Height}[1]{H({#1})}

\newcommand{\IQ}{\mathbf{Q}}

\newcommand{\torus}{\mathbf{T}}

\newcommand{\IRanexp}{{\bf R}_{\rm an,exp}}

\newcommand{\trdegS}{{\rm trdeg}}
\newcommand{\trdeg}[1]{\trdegS({#1})}

\newcommand{\atopx}[2]{\genfrac{}{}{0pt}{}{#1}{#2}}

\begin{document}
\title[]{A Note on Divisible Points of Curves}
\author{M. Bays and P. Habegger}



\begin{abstract}
Let $C$ be an irreducible algebraic curve defined  over a number field and
inside an algebraic torus of dimension at least $3$. 
  We partially answer a question posed by  Levin on points on $C$ for which a
  non-trivial power lies again on $C$. Our results have connections to
Zilber's  Conjecture on  Intersections with Tori  and yield to methods
  arising in transcendence theory and the theory of o-minimal structures.
\end{abstract}

\maketitle

\section{Introduction}
Let $C_1$ and $C_2$ be irreducible algebraic curves in the algebraic
torus $\IG_m^N$ with $N\ge 3$. 
Aaron Levin asked what can be said of  points $x$ on $C_1$ for which
 there is $n\ge 2$ such that 
$x^n$ is on $C_2$.
In this paper we give a partial answer to Levin's question in the case $C_1=C_2$.

The maximal compact subgroup of the algebraic torus is the real torus
\begin{equation*}
  \torus = \{(x_1,\ldots,x_N)\in \IG_m^N(\IC);\,\, |x_1| = \cdots =|x_N|  = 1\}. 
\end{equation*}
It is convenient to call a translate of a connected algebraic subgroup of
$\IG_m^N$ a \emph{coset} (of $\IG_m^N$). 
Moreover, a \emph{torsion coset} (of $\IG_m^N$) is a coset containing an
element of finite order. 
Torsion cosets are precisely irreducible components of algebraic
subgroups of $\IG_m^N$. 
We call a coset or torsion coset \emph{proper} if it is not equal to
$\IG_m^N$. 

Say $C\subset\IG_m^N$ is an algebraic curve defined over a number
field $F$. If $\sigma:F\rightarrow\IC$ is a field embedding, then 
$C_\sigma$ is the curve defined over $\IC$  by polynomials
obtained from applying  $\sigma$ to polynomials defining $C$. 
Let $\IQbar$ be a fixed algebraic closure of $\IQ$; we take number
fields to be subfields of $\IQbar$.

\begin{theorem}
\label{thm:main}
Let $N\ge 3$ and let
 $C\subset\IG_m^N$ be a geometrically irreducible closed algebraic curve
  defined over a number field $F$. 
We assume that $C$ is not contained in a proper torsion coset of
$\IG_m^N$.
Let us also  assume that $C_\sigma(\IC)\cap \torus$ is finite for some
 $\sigma:F\rightarrow\IC$. 
Then  $C(\IQbar)$ contains only finitely many points $x$
with $x^n\in C(\IQbar)$ for some $n\ge 2$. 
\end{theorem}

The condition $N\ge 3$ is natural for this type problem on unlikely
intersections, cf. Zannier's book  \cite{ZannierPrinceton} where Levin's question appears in
print. If $x$ is as in the theorem, then $(x,x^n)$ is contained
in the surface $C\times C$. However, it is also in an algebraic
subgroup of codimension $N > \dim (C\times C)$. Our result gives new 
evidence
 towards Zilber's Conjecture \cite{Zilber} on
intersections with tori for surfaces in $\IG_m^N$. 
The class of degenerate
subvarieties,
of which $C\times C$ is a member,
 has eluded recent 
progress by Maurin \cite{Maurin2} and the second author \cite{BHC}
towards this conjecture.
Degenerate subvarieties are defined in Maurin's work \cite{Maurin2};
an equivalent definition is $(C\times C)^{\rm oa}=\emptyset$
in Bombieri, Masser, and Zannier's notation, cf. \cite{BHC}.

However, the finiteness condition on $C_\sigma(\IC)\cap \torus$ 
does not appear in the theory of unlikely intersections, and it would
follow from Zilber's conjecture that the condition is not
necessary.
Any curve that is also a torsion coset
intersects $\torus$ in the infinite set of its points of finite
order. 
There are however also algebraic curves that are not contained in a proper
torsion coset but that intersect $\torus$ in infinitely many
points. An example is
\begin{equation*}
  \left\{\left(x,\frac{x-2}{2x-1}\right);\,\, x\in \IC\ssm\{1/2\} \right\}.
\end{equation*}
Indeed, if $|x|=1$ then $|x-2|=|2x-1|$. 

More generally, the birational map
$\IC \rightarrow \IC; z \mapsto \frac{z-1}{i(z+1)}$
maps the unit circle onto the real points  $\mathbf{P}^1(\IR)$ of the projective line, and so sets up
a bijective correspondence between curves in $\IG_m^N$ which have infinite
intersection with $\torus$ and curves in $\IC^N$ whose closure in
$(\mathbf{P}^1)^N(\IC)$ have infinite intersection with $(\mathbf{P}^1)^N(\IR)$. So
examples are plentiful.

The intersection of the line $x_1+x_2=1$ with $\torus$, however, consists
precisely of the two points
$\{(\exp(\pm 2\pi i/6),\exp(\mp 2\pi i/6)\}$.

 Corvaja, Masser, and Zannier \cite{CMZ} recently proved
 finiteness results when intersecting an algebraic curve with the maximal compact subgroup of certain commutative algebraic groups.

Our  proof involves the Theorem of
Pila-Wilkie \cite{PilaWilkie}
 which is  playing an increasingly  important role
 in diophantine problems revolving around unlikely
 intersections.  
Zannier proposed to use this tool 
  to give a new proof of the
 Manin-Mumford Conjecture for abelian varieties in joint work  with Pila \cite{PilaZannier}.
Maurin's work \cite{Maurin2} relies on a generalized Vojta inequality due to
R\'emond. Our work  uses an earlier  variant of this result
also due to R\'emond \cite{Remond:tores}, which almost immediately implies (see
Lemma~\ref{lem:vojtaapp}) that $x$ as in 
Theorem~\ref{thm:main} has height bounded only in terms of $C$. The
new ingredient in our work is the use of Baker's inequality on linear
forms in logarithms. Its effect is to obtain 
a lower bound for the degree of $x$ 
from the height bound obtained  from R\'emond's result. We refer to
Baker and W\"ustholz's  estimate \cite{BW:logforms}, which is
completely explicit. 

Using a $p$-adic version of linear forms in logarithms due to Bugeaud
and Laurent we also obtain the following partial result for curves that do not
satisfy the finiteness condition in Theorem \ref{thm:main}. 

Let $S$ be a set of rational primes. A tuple $(x_1,\ldots,x_N)$ of
algebraic numbers
 is called $S$-integral if
$\max\{|x_1|_v,\ldots,|x_N|_v\}\le 1$
for all finite places $v$ of $\IQ(x_1,\ldots,x_N)$ with
residue characteristic outside of $S$.

\begin{theorem}
\label{thm:sec}
Let $N\ge 3$ and let
 $C\subset\IG_m^N$ be a geometrically irreducible closed algebraic curve
  defined over a number field $F$. 
We assume that $C$ is not contained in a proper torsion coset of
$\IG_m^N$. There is a constant $p_0$ with the following property. 
If  $S$ is a finite set of rational primes with $\inf S > p_0$,
then  $C(\IQbar)$ contains only finitely many $S$-integral points $x$
such that $x^n\in C(\IQbar)$ for some $n\ge 2$. 
\end{theorem}

By convention the infimum of the
empty set is $+\infty$. So  $S=\emptyset$ is allowed and yields a
finiteness statement on points whose coordinates are algebraic
integers.

The paper is organized as follows. In the next section we set up common
 notation used throughout the article. The third section contains an
 elementary metric argument which is used in connection with
 Baker-type estimates in section 4. In section 5 we bound from below
 the size of Galois orbits and section 6 contains the arguments from
 o-minimality. Finally,  both theorems are proved simultaneously in
 section 7. 

Bays was partially supported by the Agence Nationale de Recherche [MODIG,
Project ANR-09-BLAN-0047], and both authors thank
Zo\'{e} Chatzidakis and the ANR for supporting Habegger's invitation 
to Paris in Summer of 2011, where the authors worked together on this
problem. Habegger is also grateful to Yann Bugeaud for answering questions in
connection with linear forms in $p$-adic logarithms.

\section{Notation} 
\label{sec:notation}
By a place $v$ of a number field $K$ we mean an absolute value that is, when
 restricted to $\IQ$, 
 either the complex absolute value or a $p$-adic absolute value for some
 prime $p$. 
A non-archimedean place is called finite and the others are called
 infinite. 
We let $K_v$ denote a completion of $K$ with respect to $v$
and, by abuse of notation, $\IQ_v$ a completion of $\IQ$ with respect
 to the restriction of $v$.

For purposes of bookkeeping it is convenient to work with height
functions. We recall here the absolute logarithmic Weil height used
throughout the article. 
Let $x\in\IQbar$. There is a unique irreducible polynomial $P \in\IZ[X]$  
with $P(x)=0$ and positive leading term $a_0$. We define the height
$x$
as
\begin{equation*}
\height{x} = \frac {1}{\deg P}\log \left(a_0\prod_{\atopx{z\in\IC}{P(z)=0}} 
\max\{1,|z|\}\right).
\end{equation*}

If $K$ is a number field containing $x$, then 
the height $\height{x}$ defined above is  
\begin{equation}
\label{eq:defineh}
\frac{1}{[K:\IQ]}\sum_{v} d_v \log \max\{1,|x|_v\}
\end{equation}
where $v$ ranges over places of $K$, and $d_v = [K_v:\IQ_v]$, the degree of
the corresponding field extension of the completions. Equivalently,
\begin{equation}
\label{eq:defineh'}
\height{x} = \frac{1}{[K:\IQ]} \sum_{p} \sum_{\sigma : K \hookrightarrow \IC_p}
\log \max\{1,|\sigma(x)|_p\}
\end{equation}
where $p$ ranges over  $\infty$ and the primes, 
$\IC_\infty=\IC$, and
$\IC_p$ is a completion of an algebraic closure of the field
of $p$-adic numbers if $p\not=\infty$.

The height of a tuple $x=(x_1,\dots,x_N)\in\IQbar^N$ is
\begin{equation*}
\height{x}=  \max\{ \height{x_1},\ldots ,\height{x_N}\}.
\end{equation*}
The exponential height of $x$ is $\Height{x}=\exp(\height{x})$.

Bombieri and Gubler's book \cite{BG} contains a thorough treatment of
heights and proofs of the claims made here and below.

\section{A Metric Argument}

The following elementary lemma is well-known and sometimes proved using Puiseux series.
We have decided to include  an elementary argument that
essentially relies only on the triangle inequality.

We let $\langle \cdot,\cdot\rangle$ denote the standard inner
product on $\IR^2$ and $\|\cdot\|$ the Euclidean norm.
If $i\in \IR^2$ we use $i_1$ and $i_2$ to denote its coordinates.
A non-zero  vector $i\in\IZ^2$ is called
reduced  
if $i_1, i_2$ are coprime and if
either $i_1 \ge 1$ or $i=(0,1)$. 
We note that two reduced vectors are  linearly dependent
if and only if they are equal.

If $K$ is a field then $K^\times$ is its multiplicative group. 
For this section we suppose that $K$ is algebraically closed
and endowed with an absolute value $|\cdot|:K\rightarrow
  \IR$.

\begin{lemma}
\label{lem:metric} 
 Suppose $P\in K[X^{\pm 1},Y^{\pm 1}]$
is a  Laurent polynomial with coefficients $p_i\in K$, $i\in\IZ^2$, and with at least $2$ non-zero terms. We set
 $D = \max_{p_ip_{i'}\not=0} \|i-i'\|$ and let
 $\sigma+1\ge 2$ denote the number of
non-zero terms of $P$ if $|\cdot|$ is archimedean 
and 
$\sigma=1$ otherwise.
There exists a finite set $\Sigma \subset  K^\times$ 
 depending only on $P$
with the following property. 
Let $x_1,x_2\in K^\times$ with $P(x_1,x_2)=0$ 
satisfy 
$\|L\| > 16D^2 (\log\sigma+\max_{p_ip_{i'}\not=0}  \log |p_i/p_{i'}|)$
where
$L = (\log |x_1|,\log |x_2|)\in\IR^2$.
Then there is $\alpha\in \Sigma$ 
and a reduced $j\in\IZ^2$ 
 with $\|j\|\le D$ such that
\begin{equation*}
  \log|x_1^{j_1}x_2^{j_2}-\alpha| \le 
-\frac{1}{16D^2} \|L\|
\end{equation*}
where we interpret the logarithm  as $-\infty$ if $x_1^{j_1}x_2^{j_2}=\alpha$.
\end{lemma}
\begin{proof}
The set $\Sigma$ and the value of $c>0$ will be determined during the
argument. 
Say $(x_1,x_2)$ is as in the hypothesis. 
We write $P= \sum_{i} p_i X^{i_1}Y^{i_2}$. After
possibly dividing $P$ by some non-zero 
term  $p_i X^{i_1}Y^{i_2}$
we may suppose that the constant term of $P$ is $1$ and 
$|p_i x_1^{i_1}x_2^{i_2}|\le 1$ for all $i\in\IZ^2$.
By abuse of notation we sometimes also write $|\cdot|$ 
for the standard absolute value on $\IR$.

 Let us fix  a non-zero element $i$ of $\IZ^2$
for which $|p_i x_1^{i_1}x_2^{i_2}|$ is maximal. 

Then $\langle L,i\rangle \le - \log |p_i|$
and we will now
 bound this quantity from
below. In the archimedean case we estimate
\begin{equation*}
1 - |p_ix_1^{i_1}x_2^{i_2}|\le   |1 + p_i x_1^{i_1}x_2^{i_2}| = \left|\sum_{i'\not=0,i}
  p_{i'}x_1^{i'_1}x_2^{i'_2}\right| \le 
\sum_{i'\not=0,i}
  |p_{i'}x_1^{i'_1}x_2^{i'_2}|\le
 (\sigma-1) |p_i x_1^{i_1}x_2^{i_2}|
\end{equation*}
by the triangle inequality and the choice of $i$.
Hence $ |p_i x_1^{i_1} x_2^{i_2}| \ge \sigma^{-1}$ 
which is also true in the non-archimedean case. 
Therefore,
 $|\langle L,i\rangle|  \le \log \sigma+|\log|p_i||<  \|L\|/(8D^2)$
by hypothesis. 
We divide by $\|i\|$ and set $v_1 = i/\|i\|$ to get
\begin{equation}
\label{eq:Lv1bound}
  |\langle L,v_1\rangle| < \frac{\|L\|}{8D^2}.
\end{equation}

We fix $v_2\in \IR^2$ of norm $1$ with $\langle v_1,v_2\rangle =0$. 
After multiplying by $-1$ we may suppose
$\langle L,v_2\rangle \ge 0$. 
We have
$\langle L,v_2\rangle^2 = \|L\|^2 - \langle L,v_1\rangle^2$
and  use (\ref{eq:Lv1bound}) to estimate
\begin{equation}
\label{eq:Lv2bound}
\langle L,v_2\rangle = |\langle L,v_2\rangle | \ge \frac{\|L\|}{2}. 
\end{equation}

Suppose $i'\in\IZ^2$ is written as
$\lambda_1
v_1 + \lambda_2 v_2$ with real numbers $\lambda_{1,2}$. Then $\lambda_2 = \langle
i',v_2\rangle$. But the coordinates of $v_2$ are up-to sign the
coordinates of $v_1 = i/\|i\|$ and so
either 
  $\lambda_2=0$  or $|\lambda_2|\ge 1/\|i\|$. 

Now suppose also
 $p_{i'}\not=0$. We  claim $\lambda_2\le
0$. 
Indeed, otherwise we would have $\lambda_2  \ge 1/\|i\|\ge 1/D$. 
Now $0 \ge \log |p_{i'}| + \langle L,i'\rangle$, so
 $ 0 \ge \log|p_{i'}| + \lambda_1 \langle L,v_1\rangle +\lambda_2
\langle L,v_2\rangle$.
We remark that $|\lambda_1|= |\langle i',v_1\rangle|\le \|i'\|\le D$.
 Using (\ref{eq:Lv1bound}) and (\ref{eq:Lv2bound})
 yields
$0 \ge \log|p_{i'}| - \|L\|/(4D) + \|L\|/(2D) 
 = \log|p_{i'}| + \|L\|/(4D)$. 
So $\|L\|\le 4D |\log |p_{i'}||$ which contradicts our
hypothesis. Thus $\lambda_2\le 0$.

We have $0 = P(x_1,x_2)= A+B$ with 
\begin{equation*}
A = \sum_{i' = \lambda_1 v_1} p_{i'} x_1^{i'_1} x_2^{i'_2}
\quad\text{and}\quad
B = \sum_{i' = \lambda_1 v_1 +\lambda_2 v_2,\, \lambda_2 <0} 
p_{i'} x_1^{i'_1} x_2^{i'_2}.
\end{equation*}

We first treat the error term $B$. If it is non-zero, 
the triangle or ultrametric triangle inequality 
yields
$ \log  |B| \le \log\sigma+ \max_{\lambda_2<0}\{\log|p_{i'}| + 
\lambda_1\langle L,v_1\rangle + \lambda_2 \langle
L,v_2\rangle\}$. 
To treat the terms in the maximum we use the bounds
(\ref{eq:Lv1bound}), (\ref{eq:Lv2bound}), 
 $|\lambda_1| \le D$, and $\lambda_2 \le -1/D$ 
proved above.
Indeed,
\begin{equation}
\label{eq:logBbound}
\log |B| \le \log\sigma + 
\max_{i'} \{\log|p_{i'}|\}
+ \frac{\|L\|}{8D} - \frac{\|L\|}{2D}
\le -\frac{\|L\|}{4D}.
\end{equation}


Now we consider the main term $A$ to which we 
associated the
polynomial
  $F = \sum_{i' = \lambda_1 v_1} p_{i'}X^{i'_1}Y^{i'_2}$.
We fix a primitive generator $j$ of the rank $1$ group
$i \IQ \cap \IZ^2$. Then 
 $j = \mu v_1$ with $\mu \in\IR$. 
We can write
 $ F = \sum_{\lambda} f_\lambda X^{\lambda j_1}Y^{\lambda j_2}$
where $\lambda$ runs over integers and the $f_\lambda$ are certain
coefficients of $P$. We define $G = \sum_{\lambda} f_\lambda
T^\lambda \in K[T^{\pm 1}]$ and with this new Laurent polynomial we have $G(X^{j_1}Y^{j_2}) =
F$. Note that $1$ is the constant term of $G$. Say $a$ is the minimal
integer such that $T^a G$ is a polynomial. 
 Let us now factor
  $T^a G = p(T-\alpha_1)\cdots (T-\alpha_d)$
where $p\not=0$ is some coefficient of $P$. We remark that
$\alpha_1,\dots,\alpha_d$ do not vanish and 
come from a finite set depending only on $P$.

For brevity we write $z = x_1^{j_1}x_2^{j_2} \in K^\times$. 
Then 
 $ p(z-\alpha_1)\cdots(z-\alpha_d) = z^a G(z) = z^a F(x_1,x_2) = -z^a B$.
Without loss of generality we may assume $|z-\alpha_1|\le
|z-\alpha_k|$ for $1\le k\le d$. So
$  |z-\alpha_1|^d \le |p|^{-1} |z|^a |B|$.
If $z=\alpha_1$ then we are in the first case of the conclusion. Else
 we take the logarithm and use  (\ref{eq:logBbound}) to estimate
\begin{equation}
\label{eq:zalpha1}
  d\log |z-\alpha_1| \le -\log|p| + a\log|z| + \log |B|
\le |\log |p|| + |a\log|z|| - \frac{\|L\|}{4D}.
\end{equation}
We continue by first bounding
 $|a\log|z|| =|a\mu \langle L,v_1\rangle|$. The inequality
$|a\mu|= |a\mu| \|v_1\|=\|aj\|\le D$ and (\ref{eq:Lv1bound}) imply
  $|a\log|z|| \le \|L\|/(8D)$.
Moreover, $|\log |p|| \le \|L\|/(16D)$.
Inserting these two inequalities into (\ref{eq:zalpha1}) yields
  $d\log|z-\alpha_1|\le -\|L\| /(16D)$.
The lemma follows since $d\le D$. 
\end{proof}


\begin{remark*}
\providecommand{\fin}{\mathcal{O}}
\providecommand{\G}{\mathbb{G}}
\providecommand{\Q}{\mathbb{Q}}
\providecommand{\N}{\mathbb{N}}
\providecommand{\Z}{\mathbb{Z}}
\providecommand{\dim}{\operatorname{dim}}
\providecommand{\trd}{\operatorname{trd}}
\providecommand{\res}{\operatorname{res}}
\providecommand{\maps}{\rightarrow} 
  Note that in the case that $|\cdot|$ is non-archimedean and all coefficients of
  $P$ have trivial absolute value, the condition on $\|L\|$ in
  Lemma~\ref{lem:metric} is just that it is non-trivial. We will use this in
  our proof of Theorem~\ref{thm:sec} below.

  Let us also remark that a qualitative version of Lemma~\ref{lem:metric},
  which does not give this information needed for Theorem~\ref{thm:sec} but
  which suffices for our uses in proving Theorem~\ref{thm:main}, admits a
  proof using the model theory of valued fields. We sketch this here.

  Let  $K$ be an algebraically closed field
  with an absolute value  $|\cdot| : K \maps \IR$.

  Define  $v : K \maps \IR\cup\{\infty\}$;  $v(x) := -\log|x|$. This is a
  valuation if and only if $|\cdot|$ is non-archimedean.
  Consider the two-sorted structure $\left<\left<K;+,\cdot\right>,\left<\IR;+,<\right>;v\right>$
  consisting of the field $K$, the ordered group $\IR$, and the map $v$.

  Let $({^*}K,{^*}\IR)$ be an elementary extension, extending $v$ to a map
  $v:{^*}K \maps {^*}\IR\cup\{\infty\}$.
  Let $\fin := \{ x \in {^*}K \;|\; \exists n\in\N.\; v(x) > -n \}$.
  Then $\fin$ is a local ring, since $v(x^{-1})=-v(x)$. Let $v'$ be the
  corresponding valuation, and let $\res$ be the corresponding residue map.
  Note that the restriction of $\res$ to $K$ is an embedding.

  (In the case that $|\cdot|$ is non-archimedean, $v'$ is the coarsening of $v$
  obtained by quotienting the value group $v({^*}K)$ by the convex hull of the
  standard value group $v(K)$. In the archimedean case with $K=\IC=\IR+i\IR$,
  we can consider $\res$ as being induced by the standard part map
  ${^*}\IR \maps \IR$)

  Now let $C \subseteq \G_m^n$ be a curve defined over $K$, and
  let $x \in C({^*}K)$.
  Suppose $||v'(x)|| := \max_i |v'(x_i)| > 0$.
  By the transcendental valuation inequality
Theorem~3.4.3  \cite{PrestelEngler},
  we have the following inequality on transcendence degrees
  of fields and dimensions of $\Q$-vector spaces:
      \[ 1 = \trd(K(x)/K) \geq \dim_\Q(v'(K(x))/v'(K)) + \trd(\res(K(x))/\res(K)) \]

  But $v'(K)=0$ and $v'(x) \neq 0$,
  so we deduce $\dim_\Q(v'(K(x))) = 1$ and $\res(K(x)) = \res(K)$.
  So say $\theta : \G_m^n \maps \G_m^{n-1}$ is an algebraic epimorphism such that
  $v'(\theta(x))=0$. Then $\res(\theta(x)) = \res(\alpha)$ for some $\alpha \in
  \G_m^{n-1}(K)$.

  Now let $\beta := \theta(x) - \alpha$. Then $\res(\beta)=0$, so $v'(\beta) \neq 0$,
  so since $\dim_\Q(v'(K(x))) = 1$, if $\beta\neq 0$
  we have $||v'(\beta)|| = q ||v'(x)||$ for some $q \in \Q$, $q>0$.

  Applying the compactness theorem of first-order logic, it follows that
  there exist finitely many pairs $(\theta,\alpha)$ of algebraic
  epimorphisms
      \[ \theta : \G_m^n \maps \G_m^{n-1} \]
  and points $\alpha \in \G_m^{n-1}(K)$,
  and there exist $q \in \Q$, $q>0$
  and $B>0$
  such that for any $x \in C(K)$
  if $||v(x)|| := \max_i(|v(x_i)|) > B$
  then for one of the finitely many pairs $(\theta,\alpha)$,
  \[ ||v(\theta(x) - \alpha)|| > q ||v(x)|| .\]
\end{remark*}

\section{Baker's Linear Forms in Logarithms}
In order to prove our theorems, we must treat archimedean and
non-archimedean places separately. We begin by proving a technical
lemma for the latter. Its proof is elementary.

\begin{lemma}
\label{lem:multdep}
For $F$  a number field, $\alpha\in F^\times$, $p$ a prime, and $B\ge 1$
 there is a constant $c>0$ with the following property.
Say $z\not=0$  is algebraic over $F$ and not a root of unity
such that $z$ and $\alpha$ are multiplicatively dependent.
We also suppose that 
$\height{z}\le B$ and
that  $|z-1|_v < 1$ for some finite place $v$ of $F(z)$ with
residue characteristic $p$.
Then $[F(z):F]+\log n \ge -c\log|z^n-\alpha|_v$ for all $n\ge 2$
for which $z^n/\alpha$ has infinite order. 
\end{lemma}
\begin{proof}
 In this lemma, the constants involved in 
Vinogradov symbols $\ll,\gg$ depend only on $F,\alpha,p,$ and $B$. 
Let $e$ be the ramification index of $v$ above $p$. 

The setup and the ultrametric triangle inequality imply $|z|_v=1$.

Let us suppose first that  $\alpha$ has finite order $m\ll 1$, say. 
We fix the integer $t\ge 0$ with $p^{t-1}\le
e/(p-1)<p^t$.
The corollary after Lemme 3 \cite{BugeaudLaurent} 
yields $|z^{p^t}-1|_v<p^{-1/(p-1)}$. 
Hence $z^{p^t}$   lies in the image of the domain of convergence
of the $p$-adic exponential function $\exp(x)=\Sigma_i x^i/i!$; cf. Chapter
II.5 \cite{Neukirch}. It follows that $|z^{kp^t}-1|_v = |k|_v |z^{p^t}-1|_v$
for any $k\in\IN$.

Hence we have
\begin{equation*}
  |mn|_v |z^{p^t}-1|_v = |z^{mnp^t}-1|_v \le |z^{mn}-1|_v  \le |z^n-\alpha|_v
\end{equation*}
and note that the very left-hand side is non-zero. 
Using $|mn|_v \ge 1/(mn)$ and taking the ordinary logarithm yields
\begin{equation}
\label{eq:logbound1}
  \log|z^{p^t}-1|_v \le \log(mn) + \log |z^n-\alpha|_v. 
\end{equation}
By the local nature  of our height (\ref{eq:defineh}),
and since $d_v \geq e$, we have
$-e\log |z^{p^t}-1|_v \le [F(z):\IQ] \height{(z^{p^t}-1)^{-1}}$. 
Basic height properties, $\height{z}\le B$, and $p^t\le
ep/(p-1)\le 2e$ now imply
\begin{equation*}
e \log |z^{p^t}-1|_v \ge - [F(z):\IQ] \height{z^{p^t}-1}
\ge -[F(z):\IQ] (\log 2 + p^t B)
\gg -[F(z):F]e.
\end{equation*}
The lemma follows if $\alpha$ has finite order after
cancelling  $e$ and using
(\ref{eq:logbound1}). 

Now we assume that $\alpha$ has infinite order.
There is a unique reduced
 $(k,l)\in\IZ^2$ such that $z^k\alpha^l=\zeta \in F(z)$ is a root
of unity. We note that $kl\not=0$ and that any pair
$(k',l')\in\IZ^2$ for which $z^{k'}\alpha^{l'}$ has finite order 
is an integral multiple of $(k,l)$. 
To complete the proof we may assume $|z^n-\alpha|_v < p^{-1/(p-1)}$.
This implies  $|\alpha|_v=1$ and this time 
we  fix $m\in\IN$ such that $|\alpha^m-1|_v < p^{-1/(p-1)}$
and $m\ll 1$. 
Thus
\begin{equation*}
  |\zeta^n- \alpha^{nl+k}|_v
= |(\zeta \alpha^{-l})^n-\alpha^k|_v
=|z^{nk}-\alpha^k|_v \le |z^n-\alpha|_v < p^{-1/(p-1)}
\end{equation*}
and passing to the $m$-th power gives
\begin{equation}
\label{eq:zetabound}
  |\zeta^{mn}-\alpha^{m(nl+k)}|_v \le |z^n-\alpha|_v < p^{-1/(p-1)}.
\end{equation} 
But $|\alpha^{m(nl+k)}-1|_v<p^{-1/(p-1)}$ holds as well. We use the
ultrametric inequality to obtain 
$|\zeta^{mn}-1|_v < p^{-1/(p-1)}$. 
The only root of unity that is $p$-adically closer to $1$ than $p^{-1/(p-1)}$
is $1$ itself. So $\zeta^{mn}=1$ and therefore
$|\alpha^{m(nl+k)}-1|_v \le |z^n-\alpha|_v$ by (\ref{eq:zetabound}). 
By our choice of $m$, the element $\alpha^m$ is in the image of the domain of convergence  of the
$p$-adic exponential function. So as above, we can  estimate
\begin{equation*}
  |nl+k|_v |\alpha^m-1|_v = |\alpha^{m(nl+k)}-1|_v \le |z^n-\alpha|_v,
\end{equation*}
so
\begin{equation*}
  |nl+k|_v \ll |z^n-\alpha|_v
\end{equation*}
since $\alpha$ is not a root of unity. 
We remark that the left-hand side is non-zero. 
Indeed,  $nl+k\not=0$ since $z^n/\alpha$ is not a root of unity

Using  $|nl+k|_v \ge |2nlk|^{-1}$
we get $|nlk|\gg |z^n-\alpha|_v^{-1} $.
So  
$\log|nlk|
 \gg -\log|z^n-\alpha|_v$
because $|nlk|\ge n\ge 2$.
The heights satisfy
 $k \height{z} = |l|\height{\alpha}$
since  $z^k = \zeta \alpha^{-l}$.
But $\alpha\not=0$ is not a root of unity. So $\height{\alpha}>0$ by
Kronecker's Theorem and therefore $|l|\ll k$. This leaves us with
 \begin{equation}
\label{eq:padicineq}
\log(nk)
\gg -\log|z^n-\alpha|_v
 \end{equation}
and to complete the proof we need to control $k$. 

The lemma follows immediately  if $k=1$.
Else,  $k\ge 2$ and  Dirichlet's Theorem from diophantine approximation
provides us with integers $k'$ and $l'$ such that $1\le k'\le k/2$ and
$|k' l-l'k|\le 2$. 
So $\height{z^{k'}\alpha^{l'}} = |k' l/k-l'|\height{\alpha} \le
2\height{\alpha}/k\ll 1/k$. On the other hand,
$z^{k'}\alpha^{l'}$ cannot have finite order since $1\le k'<k$. By a weak version of 
 Dobrowolski's Theorem \cite{Dobrowolski} we have
$\height{z^{k'}\alpha^{l'}} \gg [F(z):F]^{-2}$. 
So $k \ll [F(z):F]^2$.
The lemma follows from this inequality in combination with 
 (\ref{eq:padicineq}).
\end{proof}

In the following lemma, we apply Baker's technique of estimating linear forms
in logarithms. In the archimedean case we use the explicit estimates of Baker
and W\"ustholz \cite{BW:logforms}, and in the non-archimedean case we use a
$p$-adic version obtained by Bugeaud and Laurent \cite{BugeaudLaurent}. 

It is useful to define $L:\IG_m^N(\IC)\rightarrow \IR^N$  by
$L(x_1,\ldots,x_N) = (\log |x_1|,\ldots,\log |x_N|)$.

\begin{lemma}
\label{lem:galoisorbit1}
Let  $F$ be a number field and
 $C\subset\IG_m^2$  a geometrically irreducible algebraic curve defined over
$F$. 
We suppose that $C$ is not contained in a proper coset of $\IG_m^2$.
Let $B\ge 1$.
Let $x=(x_1,x_2)\in C(\IQbar)$ 
with $\height{x}\le B$ and  $x^n\in C(\IQbar)$
 where $n\ge 2$ is an integer. 
 \begin{enumerate}
 \item [(i)]
Say $\sigma: F(x)\rightarrow \IC$ is an embedding
and
$L = L(\sigma(x))$
then
\begin{equation}
\label{eq:degreelb}
[F(x):F] \ge c \left(\frac{n}{\log n}\|L\|\right)^{1/6}
\end{equation}
where $c>0$ depends only on $B,C,$ and $F$. 
\item[(ii)]
Say $v$ is a finite place of the number field $F(x)$ lying above a
rational prime that is sufficiently large with respect to $C$. 
If 
$(\log|x_1|_v,\log|x_2|_v)\not=0$, then 
\begin{equation}
\label{eq:degreelb2}
[F(x):F] \ge c n^{1/9}
\end{equation}
where $c>0$ depends only on $B,C,F,$ and the residue characteristic of $v$. 
 \end{enumerate}
\end{lemma}

\begin{proof}
For part (i) 
the constant $c$ is meant to be sufficiently small with respect to
$B,C,$ and $F$. It will be fixed later in the proof.
The constants implicit in  $\ll$ and $\gg$ below are positive and
 depend only on $B,C,$ and $F$. They are independent of $x$ and $c$. 
  Clearly, we may
assume
\begin{equation*}
  n\|L\| > c^{-6} \log n
\end{equation*}
as
 the conclusion (\ref{eq:degreelb}) is immediate otherwise.


The conclusion of Lemma \ref{lem:metric} applied to the point
$(\sigma(x_1)^n,\sigma(x_2)^n)$ is
\begin{equation}
\label{eq:logbound}
  \log |z^n\beta - 1| 
\ll - n \|L\|
\quad\text{where}\quad
z=\sigma(x_1^{j_1}x_2^{j_2}),
\end{equation}
 for algebraic $\beta\in \IC^\times$ 
and $j\in\IZ^2 \ssm\{0\}$, both 
coming from a finite set depending only on $C$.
After  decreasing $c$ we may assume that the left-hand side
of (\ref{eq:logbound}) is less than $-1$.

Let $l_1,l_2\in\IC$ be choices of logarithms of
$z$ and $\beta$. That is,
  $e^{l_1} = z$ and
$e^{l_2} = \beta$. 
 Some elementary calculus
yields
\begin{equation*}
 \log |\Lambda| \ll - n\|L\|
\quad\text{with}\quad
\Lambda = n  l_1 + l_2 + 2\pi i k
\end{equation*}
where $k$ is an appropriate integer with $|k|\ll n$. 

In order to  apply
  Baker's theory to bound $|\Lambda|$ from below we must first
treat the case $\Lambda=0$, so $z^n =
\beta^{-1}$.

 If $\beta$ is a not a root of
unity, then $z$ is not one either.  The same  weak version of 
Dobrowolski's Theorem as is used in Lemma \ref{lem:multdep} implies
\begin{equation}
\label{eq:dobrowolski0}
\frac{\height{\beta}}{n}=  \height{z} \gg \frac{1}{[\IQ(z):\IQ]^2}
\gg \frac{1}{[F(z):F]^2}\ge 
 \frac{1}{[F(x):F]^2}.
\end{equation}
Hence $n\ll [F(x):F]^2$ because $\beta$ is from a finite set
depending only on $P$.
The definition of the height implies $\|L\|\ll
[\IQ(x):\IQ]\height{x}$. But $\height{x}\le B$ by hypothesis,
 so $\|L\|\ll [\IQ(x):\IQ]\ll [F(x):F]$. 
We may thus bound
\begin{equation}
\label{eq:dobrowolski}
  \frac{n}{\log n}\|L\| \ll [F(x):F]n \ll [F(x):F]^3
\end{equation}
and the lemma follows in this case.

But what if $\beta$ is a root of unity? Then some positive
power of $z^n$ is $1$. Hence $x^n$ is contained in
a proper algebraic subgroup of $\IG_m^2$. But this is also a point on
$C$. Over 10 years ago Bombieri, Masser, and Zannier  \cite{BMZ}
proved that 
$\height{x^n}$ is bounded from above by a constant depending only
on $C$.
Here it is essential that $C$ is not contained in a proper coset. 
 So $\height{x}\ll 1/n$. 
By height inequalities  similar to those used above we find
$\|L\| \ll [F(x):F] \height{x} \ll [F(x):F]/n$ and
therefore
\begin{equation*}
  \frac{n}{\log n}\|L\| 
\ll [F(x):F].
\end{equation*}
We have proved (\ref{eq:degreelb})  in the case $\Lambda=0$. 

Let us now assume $\Lambda\not=0$;
we apply results on linear forms in logarithms.
An explicit version due to Baker
and W\"ustholz \cite{BW:logforms} yields
\begin{equation*}
  \log|\Lambda| \gg - [\IQ(x):\IQ]^{6} (\log A_1) (\log A_2) (\log A_3)  \log n
\end{equation*}
where $A_1,A_2,A_3$ are bounded in terms of 
 the heights of $x_1^{j_1}x_2^{j_2}$ and
$\beta$. 
But $\height{x_1^{j_1}x_2^{j_2}}\le |j_1|\height{x_1}+|j_2|\height{x_2}\ll 1$
and so we may streamline the inequality from
above to get
\begin{equation*}
  \log|\Lambda| \gg - [\IQ(x):\IQ]^{6}\log n. 
\end{equation*}
We conclude by comparing this with the upper bound for $\log|\Lambda|$
derived above. This gives
$\frac{n}{\log n} \|L\| \ll [\IQ(x):\IQ]^{6}$
and thus completes the proof of (i).

Say $v$ is as in (ii)
and let $e$ be the ramification index of $v$ above the rational prime. 

 Say $C$ is cut out by a polynomial $P$.
We may suppose that all non-zero coefficients of $P$ and all $\alpha$
provided by Lemma \ref{lem:metric} (which do not depend on 
$v$) lie in $F$ and have $v$-adic
absolute value 1.

In the proof of part (ii) we allow the constants in $\ll$ and $\gg$ to also depend
on the residue characteristic of $v$. 
 In this non-archimedean case we  only assume
$L= (\log|x_1|_v,\log|x_2|_v)\not=0$ which  implies the lower bound
$\|L\| \ge 1/e$. 
In order to complete the proof we may assume $e  < n^{1/2}$ since
$e\ge {n}^{1/2}$ implies $[F(x):F]\gg e\ge {n}^{1/2}$. 

By Lemma \ref{lem:metric}
there are reduced vectors $j,j'\in\IZ^2$ with
\begin{equation}
\label{eq:padicapprox}
\log|{z}^n-\alpha|_v \ll -\frac ne \le - n^{1/2}
\end{equation}
and $|z'-\alpha'|_v<1$
where $z=x_1^{j_1}x_2^{j_2}$, 
$z'={x_1}^{j'_1}{x_2}^{j'_2}$, and $\alpha,\alpha'$ depend only on $P$. 

Since $|\alpha|_v=|\alpha'|_v=1$ we may assume $|z|_v=|z'|_v=1$ by
(\ref{eq:padicapprox}). Moreover, $L\not=0$ so
the two equalities
\begin{equation*}
  j_1 \log |x_1|_v + j_2 \log |x_2|_v = 0
\quad\text{and}\quad
  j'_1 \log |x_1|_v + j'_2 \log |x_2|_v = 0
\end{equation*}
imply that
$j$ and $j'$ are linearly dependent. Hence $j=j'$ because they are
reduced. In particular, $z=z'$. 
There is $g\ll 1$ with $|{\alpha'}^g-1|_v<1$ and it
satisfies $|z^g-1|_v<1$ because $|z-\alpha'|_v<1$.

The exponent $g$ plays an important role in Bugeaud and Laurent's
work \cite{BugeaudLaurent}. Before applying their
 Th\'eor\`eme 3 to $z$ and $\alpha'$ we must first treat the case where
 these elements are multiplicatively dependent. 

If this is the case
 and if both $z$ and $z^n/\alpha$ have infinite order  we may apply
 Lemma \ref{lem:multdep} to $\alpha^g$ and $z^g$. 
Part (ii) follows with ample margin
 because of (\ref{eq:padicapprox}).
If $z$ has finite
 order, then
$e\|L\| \ll [F(x):F]/n$ as in the archimedean case by appealing to
 the Theorem of Bombieri, Masser, and Zannier. 
But $e\|L\| \ge 1$ and hence
$[F(x):F]\gg n$, which is better than the claim. 
If $z^n/\alpha$ has finite order and $z$ has infinite order,
then $\alpha$ has infinite order and $n\height{z}=\height{\alpha}\ll 1$. We can
 argue as near
(\ref{eq:dobrowolski0}) using a height lower bound
to again obtain $[F(x):F]\gg n^{1/2}$.

Finally, suppose  $z$ and $\alpha'$ are multiplicatively
independent. Then Bugeaud and Laurent's Th\'{e}or\`{e}me~2
\cite{BugeaudLaurent} implies
\begin{equation*}
  -\log |z^n-\alpha|_v \ll [F(z):F]^4 (\log n)^2. 
\end{equation*}
We combine this inequality with (\ref{eq:padicapprox}) to conclude the proof. 
\end{proof}

\section{Galois Orbits}

Let $N\ge 2$.
Throughout this section  $F$ is a number field and, if not stated
otherwise,
 $C\subset\IG_m^N$ is a geometrically irreducible algebraic curve defined over
$F$. 
The following lemma is a weak equidistribution  
statement on  curves.


\begin{lemma}
\label{lem:equi}
  Suppose that $C$ is not contained in a proper coset of $\IG_m^N$ and
  that there exists an embedding $\sigma_0:F\rightarrow
  \IC$ such that the curve $C_{\sigma_0}(\IC)\cap\torus$  is finite. 
For every $B\ge 1$ there exists  $\epsilon > 0$ with the following
property. 
For $x\in C(\IQbar)$ with $\height{x}\le B$
and $[F(x):F] \ge (2\#C_{\sigma_0}(\IC)\cap\torus)^N$ there is
 an embedding $\sigma:F(x)\rightarrow \IC$ extending $\sigma_0$
such that
\begin{equation*}
 \|L(\sigma(x))\| \ge \epsilon. 
\end{equation*}
\end{lemma}

\begin{proof}
By hypothesis, the intersection $C_{\sigma_0}(\IC)\cap \torus$ is
finite. 
If the intersection is empty, the
lemma is immediate. 
So we define $m=\#
C_{\sigma_0}(\IC)\cap\torus \ge 1$. 

By symmetry it suffices to prove the lemma for points
$x=(x_1,\ldots,x_N)$ as in the assertion for which
$d=[F(x_1):F]$ is  maximal among $[F(x_1):F],\ldots,[F(x_N):F]$.
So
$[F(x):F]\le d^N$ and the hypothesis on $[F(x):F]$ yields
$d \ge 2m$.

We fix $\delta \in (0,1)$ with 
\begin{equation}
\label{eq:deltachoice}
 - \log\delta = 
2[F:\IQ]^2m^2 (4B+\log 2).
\end{equation}
We take $\epsilon >0$ such that if
  $(x'_1,\ldots,x'_N)\in C_{\sigma_0}(\IC)$ and $\|L(x')\| < \epsilon$ then
$|x'_1-t|<\delta/2$ for some $t\in\IC$ appearing as a first
coordinate of a point in $C_{\sigma_0}(\IC)\cap
\torus$. The number of such values $t$ is at most $m$.

  Let us suppose that $\|L(\sigma(x))\|< \epsilon$ for all $\sigma$ as
  in the hypothesis. This will lead to a contradiction. 

 
By the Pigeonhole Principle there is
$t\in \IC$ and  a set $\Sigma$ of at least $d/m \ge
 2$ embeddings
$\sigma:F(x_1)\rightarrow \IC$ extending $\sigma_0$ such that
$|\sigma(x_1)-t|<\delta/2$.
Then $|\sigma(x_1)-\sigma'(x_1)|<\delta$ for two such
embeddings. 

The absolute value of the discriminant of the minimal polynomial of
$x_1$ is a positive integer. Its logarithm  is non-negative and hence
\begin{equation*}
  0\le  2[\IQ(x_1):\IQ]^2 \height{x_1} + \sum_{\sigma\not=\sigma'} \log|\sigma(x_1)-\sigma'(x_1)|
\end{equation*}
where $\sigma,\sigma'$ run over all embeddings $\IQ(x_1)\rightarrow
\IC$. We note that $\height{x_1}\le\height{x}\le B$. Thus
\begin{align}
\label{eq:disclowerbound}
  0&\le   2[F(x_1):\IQ(x_1)]^2[\IQ(x_1):\IQ]^2 \height{x_1} + 
 \mathcal D
\le   2  d^2 [F:\IQ]^2 B+ 
\mathcal{D}
\end{align}
with
$\mathcal D = \sum_{\sigma(x_1)\not=\sigma'(x_1)} \log|\sigma(x_1)-\sigma'(x_1)|$
 where the sum is now over all complex embeddings of
 $F(x_1)$.

Two distinct elements of $\Sigma$ take different values at $x_1$. 
Using $\log\delta <0$   we may bound
\begin{align*}
\mathcal D
&< \frac dm \left(\frac dm - 1\right) \log \delta
+ \sum_{\atopx{\sigma(x_1)\not=\sigma'(x_1)}{
\sigma\not\in\Sigma\text{ or }\sigma'\not\in\Sigma}} \log|\sigma(x_1)-\sigma'(x_1)|.
\end{align*}
Recall that $d/m\ge 2$. This yields $d/m-1\ge d/(2m)$ 
and
\begin{align*}
\mathcal D
&< 
 \frac{d^2}{2m^2} \log \delta
 + \sum_{\sigma,\sigma'}
\log(2\max\{1,|\sigma(x_1)| \}\max\{1,|\sigma'(x_1)|\}).
\end{align*}
The definition of the height 
implies
\begin{align*}
\mathcal D &<
 \frac{d^2}{2m^2} \log \delta
 + [F(x_1):\IQ]^2(2\height{x_1}+\log 2) 
\le 
 \frac{d^2}{2m^2} \log \delta
 + d^2[F:\IQ]^2(2B+\log 2)
\end{align*}
Combining this bound with (\ref{eq:disclowerbound}) and multiplying 
 by $2m^2/d^2$ yields
\begin{equation*}
-\log\delta <
4m^2[F:\IQ]^2 B + 2m^2[F:\IQ]^2(2B+\log 2) = 
2[F:\IQ]^2 m^2 (4B+\log 2)
\end{equation*}
which contradicts (\ref{eq:deltachoice}).
\end{proof}

We need a well-known preparatory lemma.
For any integer $n$ we let $[n]:\IG_m^N\rightarrow\IG_m^N$ denote the
$n$-th power homomorphism. 

\begin{lemma}
\label{lem:curvepower}
  Let $C\subset\IG_m^N$ be an irreducible algebraic curve defined over
$\IC$ which
is not a coset, $n\ge 2$ is an integer 
and $x\in \IG_m^N(\IC)$, 
then $x[n](C)\cap C$ is
  finite. 
\end{lemma}
\begin{proof}
This follows for example from Hindry's Lemme 10 \cite{Hindry:Lang}. 
\end{proof}

We use R\'emond's toric version of an inequality of Vojta in the next
lemma. 

\begin{lemma}
\label{lem:vojtaapp}
Suppose $C$ is not a coset. 
 There exists a constant $B\in\IR$ such that if
$n\ge 2$ is an integer and
$x\in C(\IQbar)$ with $x^n\in C(\IQbar)$, then $\height{x}\le B$. 
\end{lemma}
\begin{proof}
For such $n$ sufficiently large with respect to $C$ we may invoke 
 R\'emond's  Th\'eor\`eme 3.1 \cite{Remond:tores}.
In its notation we take
to $x_2 = x^n$ and $x_1=x$. 
Our result follows quickly from the observation that
R\'emond's angular distance
$\widehat{(x_1,x_2)}$ vanishes.
The finitely many integers $n\ge 2$ not covered by R\'emond's
 Theorem can be treated with Lemma \ref{lem:curvepower}. 
\end{proof}

\begin{proposition}
\label{prop:galois1}
Suppose $C$ is as in Lemma \ref{lem:equi}.
 There exists a constant $c>0$ with the following property. If
 $x\in
  C(\IQbar)$ is of infinite order  and $x^n\in C(\IQbar)$ for some
  $n\ge 2$, then
  \begin{equation*}
    [F(x):F] \ge c n^{1/7}. 
  \end{equation*}
\end{proposition}
\begin{proof}
There exists $\sigma_0$ as in Lemma \ref{lem:equi}.
 Lemma \ref{lem:vojtaapp}  implies that
 $x$ has    height bounded solely in terms
of $C$.
By Northcott's Theorem, there are only finitely many $x$
of degree at most $(2\# C_{\sigma_0}(\IC)\cap\torus)^N$. 
But for a fixed $x$ of
infinite order the
set $\{n\in\IN;\,\, x^n\in C(\IQbar) \}$ is finite by a result of
Lang, cf. Chapter 8, Theorem 3.2 \cite{FoDG}.
Here we need  that $C$ is not a coset. 
Thus for $x$ of degree at most $(2\# C_{\sigma_0}(\IC)\cap\torus)^N$, there is
an upper bound on the $n$ appearing in the statement of the lemma;
so for these $x$ the degree lower bound in the assertion is satisfied if $c$
is sufficiently small.

So we need only consider points $x$ with
$[F(x):F]\ge (2\# C_{\sigma_0}(\IC)\cap \torus)^N$. In particular,
the
degree lower bound in Lemma
\ref{lem:equi} is satisfied. 

Suppose $\epsilon > 0$ is as in said
lemma. Recall that it is independent of $x$. Now we fix
$\sigma:F(x)\rightarrow \IC$
extending $\sigma_0$ such that $\|L(\sigma(x))\|\ge \epsilon$. 

After decreasing $\epsilon$ we may assume $\|L(\sigma(x'))\|\ge
\epsilon$ where $x'$ is the projection of $x$ to two distinct
coordinates of $\IG_m^N$. This projection lies on $C'$, the Zariski
closure of the  projection 
of $C$ to $\IG_m^2$. We remark that $C'$ is 
a geometrically irreducible algebraic curve and 
 not contained in a proper coset. 
Applying Lemma \ref{lem:galoisorbit1}(i) to $C'$,
inequality (\ref{eq:degreelb})  
yields
\begin{equation*}
  c\epsilon\frac{n}{\log n}  \le c\frac{n}{\log n}\|L(\sigma(x'))\|
\le [F(x'):F]^{6} \le  [F(x):F]^{6},
\end{equation*}
with $c>0$ independent of $x$ and $n$. The proposition follows after
adjusting $c$. 
\end{proof}

\begin{lemma}
\label{lem:integral}
  Suppose that $C$ is not a torsion coset of $\IG_m^N$.
There exists $\epsilon > 0$ with the following property. 
If 
 $x=(x_1,\ldots,x_N)\in C(\IQbar)$ has infinite order 
and if all $x_i$ are algebraic integers
there is
 an embedding $\sigma:F(x)\rightarrow \IC$
with
\begin{equation*}
 \|L(\sigma(x))\| \ge \epsilon. 
\end{equation*}
\end{lemma}
\begin{proof}
The Bogomolov Conjecture, proved in this case by
Zhang \cite{ZhangArVar}, implies that there is $\epsilon>0$
such that any
algebraic point of $C$ of infinite order has height at least $\epsilon$. 
If $x$ is as in the hypothesis then only the archimedean places
contribute to the height of $x$. The lemma follows easily after possibly
modifying $\epsilon$.
\end{proof}

\begin{proposition}
\label{prop:galois2}
Suppose $C$ is as in Lemma \ref{lem:integral}.
Let  $S$ be a finite set of rational primes with $\inf S$ sufficiently
large with respect to $C$. 
There exists a constant $c>0$ with the following property. 
If
 $x=(x_1,\ldots,x_N)\in
  C(\IQbar)$ is $S$-integral and of infinite order  with $x^n\in C(\IQbar)$ for some
  $n\ge 2$, then
  \begin{equation*}
    [F(x):F] \ge c n^{1/9}.
  \end{equation*}
\end{proposition}
\begin{proof}
The proposition follows along the lines of the proof of 
Proposition \ref{prop:galois1}. 
Indeed, if all $x_i$ are algebraic integers, then
Lemmas \ref{lem:galoisorbit1}(i)  and \ref{lem:integral}
do the trick with the better exponent $1/6$ . 
Otherwise there is a number field and  a place $v$ with residue characteristic in $S$
such that $|x_i|_v>1$ for some $i$. We may suppose that
the said characteristic is large enough in order to apply part (ii) of
Lemma \ref{lem:galoisorbit1} to a suitable projection of $C$. 
Inequality (\ref{eq:degreelb2}) concludes the proof. 
\end{proof}

\section{o-minimality}
\label{sec:ominimal}
\providecommand{\IRexpsin}{\IR_{\exp,\sin_|}}

We refer to \cite{D:oMin} for the essentials on o-minimal structures. 
We will work with
the  structure $\IRexpsin = \left<\IR;+,\cdot,\exp,\left.\sin\right|_{[0,2\pi)}\right>$ generated as a
structure over the real field by the real exponential function and restricted
sine. This structure is o-minimal, by a theorem of Wilkie, cf.
Example B \cite{Wilkie:96}; its o-minimality also follows from the well-known
stronger result that $\IRanexp$ is o-minimal.

We will call a subset of $\IR^m$
 definable if it is definable in $\IRexpsin$. We identify $\IC$ with
 $\IR^2$ by identifying a complex number with the pair consisting of
 its real and imaginary parts.

Let $C\subset\IG_m^N$ be an irreducible algebraic curve defined over $\IC$.
The complex exponential function $\exp : \IC^N\rightarrow(\IC^\times)^N$
restricted to the fundamental domain 
$\mathcal F = (\IR + [0,2\pi)i)^N$ is definable. Hence
\begin{equation*}
  X = \{(n,z,k)\in\IR \times\mathcal F \times\IR^N;\,\, 
\exp(z)\in C(\IC),\; nz-2\pi ik\in \mathcal F,\;
\exp(nz-2\pi i k)\in C(\IC) \},
\end{equation*}
is also definable,
as is its projection $Z\subset \IR \times \IR^{N}$
to the first and the last $N$ coordinates. 

We will treat $X$ and $Z$ as definable families parametrized by 
the
parameter $n\in\IR$. As is usual $X_n\subset \mathcal F \times \IR^N$
and $Z_n\subset \IR^N$ denote the respective fibers. 
The latter  is the projection of
$X_n$ to the final $N$ coordinates. 

We will be interested in integral
$n$. We will show that $Z_n$ contains no semi-algebraic curves if
$n\ge 2$.

The main tool in studying the semi-algebraic curves contained in
fibers of $Z$ is Ax's Theorem. We state only a special case. 
Let $\Delta = \{t\in\IC;\,\, |t|<1\}$.

\begin{theorem}[Ax]
  Suppose $\gamma_1,\ldots,\gamma_N : \Delta\rightarrow \IC$ are
  holomorphic functions   for which
$\gamma_1-\gamma_1(0),\ldots,\gamma_N-\gamma_N(0)$ are 
$\IZ$-linearly independent. Then 
  \begin{equation*}
    \trdeg{\IC(\gamma_1,\ldots,\gamma_N,\exp\circ\gamma_1,\ldots,\exp\circ\gamma_N)/\IC}
    \ge N+1.
  \end{equation*}
\end{theorem}
\begin{proof}
By hypothesis, not all $\gamma_i$ are constant. So the claim 
 follows directly from the one variable case of Corollary~2 \cite{Ax:Schanuel}.
\end{proof}

\begin{proposition}
\label{prop:nocurves}
We  assume that $N\ge 3$ and that $C$ is not contained in a proper coset of
  $\IG_m^N$.  If $n\ge 2$ is integral,  then $Z_n$ does not contain a semi-algebraic curve. 
\end{proposition}

\begin{proof}
  Let us assume the contrary and suppose that $Z_n$ contains a
  semi-algebraic curve $S$. 
So there is a semi-algebraic and non-constant map
 $\gamma:(-1,1)\rightarrow S$. After reparametrizing we may suppose
that the coordinates $\gamma_1,\ldots,\gamma_N$ of $\gamma$ extend to
holomorphic functions defined on $\Delta$
with values in $\IC$
and with 
$\trdeg{\IC(\gamma_1,\ldots,\gamma_N)/\IC}=1$.

 ``Definable Choice''
provides us with 
a definable map $\theta:(-1,1)\rightarrow \mathcal F$ such that
\begin{equation*}
  (\theta(t),\gamma(t)) \in X_n\quad\text{for all}\quad t\in (-1,1).
\end{equation*}

Definable sets in $\IRexpsin$ admit a decomposition into analytic cells.
So the function $\theta$ is real analytic in a neighborhood
of some point of $(-1,1)$. 
At this point the real- and imaginary parts of all coordinates of
$\theta$ have 
a Taylor expansion.
Each of these $2N$ functions admits a holomorphic continuation in  a
neighborhood of said point. 
By taking appropriate linear combinations after possibly reparametrizing, we find that
 $\theta$ continues to 
 a holomorphic
function
$\Delta\rightarrow \IC^N$. The image of
$\exp\circ\theta:\Delta\rightarrow\IG_m^N(\IC)$ is still contained in
$C(\IC)$ by analyticity.

(Although we choose not to present it this way, readers versed in model theory
may find it helpful to consider $\gamma(t)$ and $\theta(t)$ as points in an
elementary extension ${^*}\IR = \operatorname{dcl}(\IR,t)$, with $\gamma(t)$
generic on $S$ over $\IR$ and $\theta(t)$ a witness point in $X_n$. In this
view, we can directly apply Theorem 3 \cite{Ax:Schanuel} via a technique due
to Wilkie, considering ${^*}\IC={^*}\IR+i{^*}\IR$ as a differential field with
$\delta(f(t)):=f'(t)$ for $f$ a function defined over $\IR$.)

We consider the morphism
$\varphi:\IG_m^N\times\IG_m^N\rightarrow\IG_m^N$ defined by
$\varphi(x,y) = x^{-n}y$. The Zariski closure of $\varphi(C\times C)$ has dimension
at most $2$.  The Zariski closure $Y$ of the image of
 $t\mapsto \exp(-2\pi i \gamma(t))$ is
contained in $\varphi(C\times C)$ and hence has
dimension at most $2$. We remark that $Y$ is irreducible.

Since $N+1 > 1 + 2$, Ax's Theorem implies that the
co-ordinates of $\gamma$ are $\IZ$-linearly dependent modulo constants, and so
the variety $Y$ is contained in a proper coset.

If $\dim Y = 2$ then 
 $\varphi(C\times C)$ is contained in a proper coset.
The same then holds true for 
 $C\times C$ and also
 $C$. This contradicts our hypothesis
and so $\dim Y\le 1$. 

So if  $1\le i,j\le N$ then 
\begin{equation*}
  \trdeg{\IC
  (\gamma_i,\gamma_j,\exp\circ\gamma_i,\exp\circ\gamma_j) /\,\IC} \le
  \trdeg{\IC(\gamma_1,\ldots,\gamma_N,\exp\circ\gamma_1,\ldots,\exp\circ\gamma_N) /\IC}  \le 2
\end{equation*}
and so $\gamma_i$ and $\gamma_j$ are $\IZ$-linearly dependent 
 modulo constants by Ax's Theorem. So $Y$ is itself a coset, 
 and there is a
surjective homomorphism of algebraic groups $\phi
:\IG_m^N\rightarrow \IG_m^{N-1}$
that is constant on $Y$. 

It is convenient to write 
$C'$ for the Zariski closure of $\phi(C)$.
For $t\in \Delta$ we have $\phi(\exp(n\theta(t)-2\pi i \gamma(t)))\in 
C'(\IC)$. Moreover,
\begin{equation*}
  \phi(\exp(n\theta(t)-2\pi i \gamma(t)))
 = x_0 \phi(\exp\theta(t))^n \in C'(\IC)\cap x_0 [n](C'(\IC))
\end{equation*}
where $x_0  = \phi(\exp(-2\pi i \gamma(t)))$ is independent of $t$. 
  
If $t\mapsto \phi(\exp\theta(t))^n$ is non-constant on $\Delta$, then 
it takes infinitely many values. In this case $C' \cap x_0
[n](C')$ is infinite.
By comparing dimension we find
$C' = x_0 [n](C')$
since both sides are irreducible curves. 
%
By Lemma \ref{lem:curvepower} the curve $C'$
is a coset. Therefore, $C$ is contained in a proper coset of
$\IG_m^N$ and we have a contradiction.

So $t\mapsto \phi(\exp \theta(t))^n$ 
is  constant. If $t\mapsto \exp \theta(t)$ is non-constant,  its
image is in a translate of the 1-dimensional kernel of $\phi$. 
This too is impossible as $\exp\theta(t)\in C(\IC)$. Hence
$t\mapsto \exp\theta(t)$  is constant and the same holds for  $t\mapsto \theta(t)$.
As we deduced from Ax's Theorem, the components of $\gamma$ are $\IZ$-linearly
dependent modulo constants. 
But $\gamma$ is non-constant and holomorphic, so 
 $\exp(n\theta -2\pi i \gamma)$ takes infinitely many 
values in $C(\IC)$. 
It follows that  $C$ is contained in a proper coset
and this contradicts the hypothesis. 
\end{proof}

\section{Proof of the Theorems}

We will prove Theorems \ref{thm:main} and \ref{thm:sec} simultaneously.
The only difference in treating  these two cases is if we will refer
to Proposition  \ref{prop:galois1} or 
 \ref{prop:galois2}.

Suppose that $C$ is as in the hypothesis.
We first reduce to the case where $C$ is not contained in a proper
coset. Otherwise there is a 
 coset $yH$ with $C\subset yH$ where $H\subsetneq\IG_m^n$ is a connected algebraic
subgroup. If $x\in C(\IQbar)$ and $x^n\in C(\IQbar)$ for some $n\ge
2$, then $x\in yH\cap y^nH$ and so $yH = y^nH$ which yields
$y^{n-1}\in H$. Therefore, $yH$ is a proper torsion coset
containing $C$, contradicting the hypothesis. 

The Manin-Mumford Conjecture for curves  in $\IG_m^N$,
Lang \cite{LangDivision} presents proofs due to Ihara, Serre, and Tate,
 states that there are only
finitely many points on $C$ of finite order. 
Hence we must prove that there are only finitely many non-torsion
$x\in C(\IQbar)$ with $x^n\in C(\IQbar)$ for some $n\ge 2$. 
For a fixed $n$ there are only finitely many such points
by Lemma \ref{lem:curvepower}. 
We will prove the theorem by  deriving
 a contradiction for all sufficiently large $n$. 

By Proposition \ref{prop:galois1}
 or \ref{prop:galois2}, there are at least $c_1 n^{1/9}$
conjugates $x_1,\ldots,x_M$ of $x$ over $F$; here and
below $c_1,c_2,c_3$ are positive constants that do not depend on
$n$ or $x$.

 These conjugates and their $n$-th
powers are points on $C(\IQbar)$. 
For each $x_j$ there is $z_j \in \mathcal F$ and $k_j \in \IZ^N$ with 
$\exp(z_j) = x_j$ and $n z_j - 2\pi i k_j \in \mathcal F$. 
Then 
\begin{equation}
\label{eq:xjkj}
  (z_j,k_j) \in X_n \quad\text{and}\quad k_j \in Z_n \cap\IZ^N \quad\text{for all}\quad
1\le j\le M
\end{equation}
where $X$ and $Z$ are as in Section \ref{sec:ominimal}.

The fiber $Z_n$ does not contain a real semi-algebraic curve by
Proposition \ref{prop:nocurves}. So
we can use the Pila-Wilkie Theorem \cite{PilaWilkie} to bound from above the number of
integral points of bounded height on $Z_n$. 
As this result is uniform
in families we find a constant $c_2$ independent of $n$ such that
\begin{equation}
\label{eq:PW}
  \# \{ k \in Z_n \cap \IZ^N;\,\, \Height{k} \le n \}
\le c_2 n^{1/10};
\end{equation}
 we recall that the height $\Height{k}$ was defined
in Section \ref{sec:notation}.
Of course, the exponent $1/10$ can be replaced by any positive
constant at the cost of increasing $c_2$. 

 We proceed to show that $k_j$ are among the elements being counted in 
  (\ref{eq:PW}). Indeed,
the imaginary part of the components of $z_j$ and $nz_j - 2\pi i k_j$ lie in
$[0,2\pi)$. So the components of $k_j$ are in $[0,n)$. Because
    $\Height{k_j}$ is the largest modulus of a component
    we find $\Height{k_j}\le n$. 

We  already know that there at least $c_1 n^{1/9}$ distinct
$z_j$. But we require a lower bound for the number of $k_j$. 
By a  basic uniformity property of o-minimal structures the number
of connected components of the fiber
\begin{equation*}
  X_{n,k_j} = \{ z\in \mathcal F;\,\, (z,k_j)\in X_n\}
\end{equation*}
is bounded from above by a constant $c_3$ that does not dependent on
$n$ or $k_j$.

We claim that each $X_{n,k_j}$ is finite. If this is not the case  there
would be infinitely many $x\in C(\IC)$ with $x^n \in C(\IC)$.
We have already seen that this is impossible by Lemma \ref{lem:curvepower}.

So $\# X_{n,k_j}\le c_3$ and the number of distinct $k_j$ is at least
$\frac{c_1}{c_3}n^{1/9}$. This lower bound contradicts (\ref{eq:PW})
 for $n$ sufficiently large. Thus $n$ is bounded from above and therefore we
 deduced the desired finiteness result. 
\qed

\bibliographystyle{amsplain}

\def\cprime{$'$}
\providecommand{\bysame}{\leavevmode\hbox to3em{\hrulefill}\thinspace}
\providecommand{\MR}{\relax\ifhmode\unskip\space\fi MR }
\providecommand{\MRhref}[2]{%
  \href{http://www.ams.org/mathscinet-getitem?mr=#1}{#2}
}
\providecommand{\href}[2]{#2}

\vfill
\address{
\noindent Martin Bays,  McMaster University, Ontario, Canada,
{\tt mbays@sdf.org}

\vspace{0.5cm}

\noindent
Philipp Habegger,
TU Darmstadt,
Schlossgartenstrasse 7,
Darmstadt,
64289 Germany,
{\tt habegger@mathematik.tu-darmstadt.de}
}
\bigskip
\hrule
\medskip

\end{document}